\documentclass[11pt,psamsfonts]{amsart}
\usepackage{amsmath}
\usepackage{amsthm}
\usepackage{amssymb}
\usepackage{amscd}
\usepackage{amsfonts}
\usepackage{amsbsy}
\usepackage{graphicx}
\usepackage[dvips]{psfrag}
\usepackage{array}
\usepackage{color}
\usepackage{epsfig}
\usepackage{url}
\usepackage{ulem}
\usepackage{comment}

\newcolumntype{L}{>{\displaystyle}l}
\newcolumntype{C}{>{\displaystyle}c}
\newcolumntype{R}{>{\displaystyle}r}

\newcommand{\R}{\ensuremath{\mathbb{R}}}

\newcommand{\Z}{\ensuremath{\mathbb{Z}}}

\newcommand{\CO}{\ensuremath{\mathcal{O}}}
\newcommand{\ov}{\overline}

\newcommand{\T}{\theta}
\newcommand{\f}{\varphi}
\newcommand{\al}{\alpha}

\newcommand{\X}{\ensuremath{\mathcal{X}}}

\def\p{\partial}
\def\e{\varepsilon}

\newtheorem {theorem} {Theorem}
\newtheorem {proposition} [theorem] {Proposition}
\newtheorem {corollary} [theorem] {Corollary}
\newtheorem {lemma} [theorem] {Lemma}

\newtheorem {remark} {Remark}
\newtheorem {claim} {Claim}
\newtheorem {mtheorem} {Theorem}

\begin{document}

\title[Bifurcation of periodic solutions in piecewise differential systems]
{Bifurcations from families of periodic solutions\\
 in piecewise differential systems}

\author[J. Llibre, D.D. Novaes and C.A.B. Rodrigues]
{Jaume Llibre, Douglas D. Novaes  and Camila A. B. Rodrigues}

\let\thefootnote\relax\footnotetext{\\ Jaume Llibre\\
Departament de Matematiques,
Universitat Aut\`{o}noma de Barcelona, 08193 Bellaterra,\\
Barcelona, Catalonia, Spain\\
Email: jllibre@mat.uab.cat}

\let\thefootnote\relax\footnotetext{\\ Douglas D. Novaes\\
Departamento de Matem\'{a}tica, Universidade Estadual de Campinas,\\
Rua S\'{e}rgio Buarque de Holanda, 651, Cidade Universit\'{a}ria Zeferino Vaz, 13083-859,\\
Campinas, S\~{a}o Paulo, Brazil,\\
Email: ddnovaes@unicamp.br}

\let\thefootnote\relax\footnotetext{\\ Camila A. B. Rodrigues\\
Departamento de Matem\'{a}tica, Universidade Federal de Santa Catarina, 88040-900,\\
Florian\'{o}polis, Santa Catarina, Brazil,\\
Email: camila.rodrigues.math@gmail.com} \maketitle

\noindent{\bf Abstract.} Consider a differential system of the form
\[
x'=F_0(t,x)+\sum_{i=1}^k \e^i F_i(t,x)+\e^{k+1} R(t,x,\e),
\]
where $F_i:\mathbb S^1 \times D \to \R^m$ and $R:\mathbb S^1 \times D
\times (-\e_0,\e_0) \to \R^m$ are piecewise $C^{k+1}$ functions and
$T$-periodic in the variable $t$. Assuming that the unperturbed
system $x'=F_0(t,x)$ has a $d$-dimensional submanifold of periodic
solutions with $d<m$, we use the Lyapunov-Schmidt reduction  and
the averaging theory to study the existence of isolated $T$-periodic solutions of the
above differential system.

\smallskip


\smallskip



\section{Introduction and Statement of the main result}

\subsection{Introduction}

The study of invariant sets, in special isolated periodic
solutions, is very important for understanding the dynamics of a
differential system. 
In the present study, we are concerned about isolated $T$-periodic solutions of non-autonomous differential systems written in the form
\begin{equation}\label{initialsystem11}
x'=F(t,x;\e)=F_0(t,x) + \sum_{i=1}^k \e^i F_i(t,x)+\e^{k+1} R(t,x,\e),\, (t,x)\in \R \times D.
\end{equation}
Here, the prime denotes the derivative with respect to the
independent variable $t,$ all the functions are assumed to be $T$-periodic in $t,$ $D$ is an open subset of $\R^m,$ and $\e$ is a small parameter. In this regard, the
averaging theory serves as an important tool to detect periodic
solutions of \eqref{initialsystem11}. A classical introduction to the averaging theory can be found in
\cite{versanmur,ver}. 

There are many studies concerning the periodic solutions of system \eqref{initialsystem11}.  As a fundamental hypothesis, it is usually assumed that the unperturbed system $x'=F_0(t,x)$ has a submanifold of initial conditions $\mathcal{Z}\subset D$ whose orbits are $T$-periodic. 
These studies differ among them depending on the regularity of system \eqref{initialsystem11} and on the dimension of $\mathcal{Z}.$ In what follows, we shall quote some of them.

For the case $\dim(\mathcal{Z})=m,$ the classical averaging theory \cite{versanmur,ver} provides sufficient conditions for the existence o periodic solutions of \eqref{initialsystem11} assuming $F_0=0$ and some  smoothness and boundedness conditions. In \cite{BuicaLlibre}, the authors extended the former results up to $k=2$ assuming weaker conditions on the regularity of system \eqref{initialsystem11}. In \cite{GGL}, the authors dropped the condition $F_0=0$ and developed the averaging theory at any order ($k \geq 1$ being an arbitrary integer) assuming the analyticity of the system \eqref{initialsystem11}.  The analyticity condition was relaxed in \cite{LliNovTeiN2014} by means of topological methods. The averaging theory was also extended to non-smooth differential
systems \cite{LliItiNov2015,LliMerNovJDE2015,LliNovPreprint2014,LliNovRod,LliNovTeiBSM2015}. The study of non-smooth differential systems is important in many fields of applied sciences since many problems of
physics, engineering, economics, and biology are modeled using
differential equations with discontinuous right-hand side, see for
instance \cite{BBCK,Co,physDspecial}. Thus, there is natural interest in studying the periodic solutions of system \eqref{initialsystem11}. when it is not smooth.

For the case $\dim(\mathcal{Z})<m$, the averaging theory by itself is not enough
to analyze  the periodic solutions of system \eqref{initialsystem11} and other
techniques need to be employed with it, such as the
\textit{Lyapunov-Schmidt reduction method}. In the case that $F_i$'s
are smooth functions, we may quote the studies \cite{BuicaLlibreFran, LliNovCand, GineLlibreZhang, jaumedouglas}. If the functions $F_i$ are not smooth or even continuous, we have studies
\cite{LliNovPreprint2014,LliNovZel}, where the authors analyzed
some classes of these systems.

In the sequel, we describe how the averaging theory and 
Lyapunov-Schmidt reduction are used for computing isolated periodic solutions of  piecewise smooth differential systems.

\subsection{Lyapunov-Schmidt reduction} \label{1.2} Consider the function

\begin{equation}\label{functiong1} g(z,\e)= \sum_{i=0}^{k}\e^{i} g_i(z) +
\CO(\e^{k+1}),
\end{equation}
where $g_i:D \to \R^m$ is a $C^{k+1}$ function, $k\geq 1$, for
$i=0,1,\ldots,k$, and $D$ is an open bounded subset of $\R^m$. For
$d<m$, let $V$ be an open bounded subset of $\R^d$ and
$\beta:\overline{V} \to \R^{m-d}$ a $C^{k+1}$ function such that
\begin{equation} \label{zdefi} \mathcal{Z} =
\{z_{\alpha}=(\alpha,\beta(\alpha)) : \alpha \in \overline{V}\}
\subset D.
\end{equation} 
We assume that
\begin{itemize}
\item[$(H_{a}$)] {\it function $g_0$ vanishes on the $d$-dimensional
submanifold $\mathcal{Z}$ of $D$.}
\end{itemize}

In \cite{LliNovCand}, the authors used the Lyapunov-Schmidt reduction
method to develop the \textit{bifurcation function of order $i$},
for $i=0,1,\ldots,k$, which for $|\e| \neq 0,$ sufficiently small,
control the existence of branches of zeros $z(\e)$ of system
\eqref{functiong1} that bifurcate from $z(0) \in \mathcal{Z}$. In
this subsection we present the results developed in that study and
those that we shall need later on. To do this we need to introduce some notations.

 Consider the projections onto the
first $d$ coordinates and onto the last $m-d$ coordinates denoted by
$\pi:\R^d \times \R^{m-d} \to \R^d$ and $\pi^{\perp}:\R^d \times
\R^{m-d} \to \R^{m-d}$, respectively. In addition, for a point $z \in
\mathcal{Z}$ we write $z=(a,b) \in \R^d \times \R^{m-d}$.

Let $L$ be a positive integer, let $x=(x_1,x_2,\ldots,x_m) \in D$,
$t\in\R$ and $y_j= (y_{j1},\ldots, y_{jm})\in \R^m$ for
$j=1,\ldots,L$. Given $G:\R\times D\rightarrow\R^m$ is a sufficiently
smooth function, for each $(t,x)\in\R\times D$ we denote by
$\p^LG(t,x)$ a symmetric $L$--multilinear map which is applied to a
``product'' of $L$ vectors of $\R^m$, which we denote as
$\bigodot_{j=1}^Ly_j\in \R^{mL}$. The definition of this
$L$--multilinear map is
\begin{equation*}\label{p}
\p^L G(t,x)\bigodot_{j=1}^Ly_j= \sum_{i_1,\ldots,i_L=1}^n \dfrac{\p^L
    G(t,x)}{\p x_{i_1}\ldots \p x_{i_L}}y_{1i_1}\ldots y_{Li_L}.
\end{equation*}
We define $\p^0$ as the identity.

The bifurcation  functions $f_i: \overline{V} \to \R^d$ of order $i$
are defined for $i=0,1,\ldots,k$ as
\begin{equation}
\label{bifurcationfunction} f_i(\al)= \pi g_i(z_{\al}) +
\sum_{l=1}^{i} \sum_{S_l} \dfrac{1}{c_1!\,c_2!2!^{c_2}\ldots
c_l!l!^{c_l}}\p_b^L \pi g_{i-l}(z_{\al})
\bigodot_{j=1}^l\gamma_j(\alpha)^{c_j},
\end{equation}
where the $\gamma_i : V \to \R^{m-d}$, for $i=1,2,\ldots, k$, are
defined recursively as
\begin{equation}
\label{gammadefi}
\begin{array}{RL}
\gamma_1(\alpha)=& -\Delta_{\al}^{-1} \pi^{\perp} g_1(z_{\al}) \quad \text{and}\\
\gamma_i(\alpha)=&-i! \Delta_{\al}^{-1} \bigg(\sum_{S'_i}
\dfrac{1}{c_1!\,c_2!2!^{c_2}\ldots
c_{i-1}!(i-1)!^{c_{i-1}}}\p_b^{I'} \pi^{\perp} g_{0}(z_{\al})
\bigodot_{j=1}^{i-1}\gamma_j(\alpha)^{c_j} \\
& +\sum_{l=1}^{i-1} \sum_{S_l} \dfrac{1}{c_1!\,c_2!2!^{c_2}\ldots
c_{l}!l!^{c_{l}}}\p_b^{L} \pi^{\perp} g_{i-l}(z_{\al})
\bigodot_{j=1}^{l}\gamma_j(\alpha)^{c_j}\bigg).
\end{array}
\end{equation}
We denote by $S_l$ the set of all $l$-tuples of non-negative integers
$(c_1,c_2, \ldots,c_l)$ such that $c_1+2c_2+\ldots+lc_l=l$,
$L=c_1+c_2+\ldots+c_l$, and by $S'_i$ the set of all $(i-1)$-tuples
of non-negative integers $(c_1,c_2, \ldots,c_{i-1})$ such that
$c_1+2c_2+\ldots+(i-1)c_{i-1}=i$, $I'=c_1+c_2+\ldots+c_{i-1}$ and
$\Delta_{\al}=\dfrac{\p \pi^{\perp} g_0}{\p b}(z_{\al})$.

Concerning the zeros of function \eqref{functiong1}, the following result was  proven in \cite{LliNovCand}:

\begin{theorem}[\cite{LliNovCand}, Corollary 1]
\label{LStheorem} Let $\Delta_{\al}$ denote the lower right corner
$(m-d)\times (m-d)$ matrix of the Jacobian matrix $D g_0(z_{\al})$.
In addition to hypothesis $(H_{a}),$ we assume that
\begin{itemize}
\item[(i)] for each $\al \in \overline{V}$, $\det \Delta_{\al} \neq 0$; and
\item[(ii)] $f_1=f_2=\ldots=f_{k-1}=0$ and $f_k$ is not identically zero.
\end{itemize}
If there exists $\al^{\ast} \in V$ such that $f_k(\al^{\ast})=0$ and
$\det(Df_k(\al^{\ast})) \neq 0$, then there exists a branch of zeros
$z(\e)$ with $g(z(\e),\e)=0$ and
$|z(\e)-z_{\al^{\ast}}|=\mathcal{O}(\e)$.
\end{theorem}

\subsection{\textbf{The averaged functions}}

Let $k \geq 1$ be a positive integer. In \cite{LliNovCand}, the authors have used Theorem \ref{LStheorem} to provide sufficient conditions for the existence of isolated periodic solutions of the following $T$-periodic non-autonomous differential system
\begin{equation}
\label{continoussystem} x'=F(t,x,\e)=F_0(t,x)+\sum_{i=1}^k \e^i
F_i(t,x) + \CO(\e^{k+1}), \quad (t,z) \in \R \times D,
\end{equation}
Here, $F_i$'s are $C^{k+1}$ functions $T$-periodic in the variable $t$.  We assumed that 
\begin{itemize}
\item[$(H_{b}$)] {\it all the solutions of the unperturbed system
$
x'=F_0(t,x)
$
starting at points of $\mathcal{Z},$ defined in \eqref{zdefi},   are $T$-periodic.}
\end{itemize}

Consider the variational equation
\begin{equation}\label{cvareq}
y'= \dfrac{\p F_0}{\p x}(t,x(t,z,0))y,
\end{equation}
where $x(t,z,0)$ denotes the solution of system
\eqref{continoussystem} when $\e=0$. Denote a fundamental
matrix of system \eqref{cvareq} by $Y(t,z)$. The \textit{average
function of order $i$} of system \eqref{continoussystem} is defined
as
\begin{equation}\label{averagefunctions}
g_i(z)= Y^{-1}(T,z)\frac{y_i(T,z)}{i!},
\end{equation} where
\begin{equation}\label{y}
\begin{array}{RL}
y_1(t,z)=&Y(t,z)\int_0^{t} Y(s,z)^{-1} F_1\left(s,x(s,z,0)\right)ds,\vspace{0.3cm}\\
y_i(t,z)=&i!Y(t,z)\int_0^{t}Y(s,z)^{-1}\Big(F_i\left(s,x(s,z,0)\right)\\
&+\sum_{S'_i} \dfrac{1}{b_1!\,b_2!2!^{b_2}\ldots
b_{i-1}!(i-1)!^{b_{i-1}}}\p^{I'} F_0(s,x(s,z,0)) \bigodot_{j=1}^{i-1}y_j(s,z)^{b_j} \\
& +\sum_{l=1}^{i-1} \sum_{S_l} \dfrac{1}{b_1!\,b_2!2!^{b_2}\ldots
b_{l}!l!^{b_{l}}}\p^{L} F_{i-l}(s,x(s,z,0))
\bigodot_{j=1}^{l}y_j(s,z)^{b_j}\bigg)ds.
\end{array}
\end{equation}

Now, in the bifurcation functions formulae \eqref{bifurcationfunction}, let the functions $g_i$'s be given by \eqref{averagefunctions} instead of \eqref{functiong1}. Concerning the isolated periodic solutions of system \eqref{continoussystem}, the following result was  proven in \cite{LliNovCand}: 

\begin{theorem}[\cite{LliNovCand}, Corollary 2]
\label{ATtheorem} Let $\Delta_{\al}$ denote the lower right corner
$(m-d)\times (m-d)$ matrix of the Jacobian matrix $Y(0,z_{\al})^{-1}-Y(T,z_{\al})^{-1}$.
In addition to hypothesis $(H_{b}),$ we assume that
\begin{itemize}
\item[(i)] for each $\al \in \overline{V}$, $\det \Delta_{\al} \neq 0$; and
\item[(ii)] $f_1=f_2=\ldots=f_{k-1}=0$ and $f_k$ is not identically zero.
\end{itemize}
If there exists $\al^{\ast} \in V$ such that $f_k(\al^{\ast})=0$ and
$\det(Df_k(\al^{\ast})) \neq 0$, then there exists a $T$-periodic solution $\varphi(t,\e)$ of \eqref{continoussystem} such that $|\varphi(0,\e)-z_{\al^{\ast}}|=\mathcal{O}(\e)$.
\end{theorem}

\begin{remark}\label{remark}
The functions $y_i(t,z)$ are obtained recurrently as solutions of the following 
integral equations 
\begin{equation}
\label{y1}
\begin{array}{RL}
y_1(t,z)=&\int_0^{t} \Big( F_1\left(s,x(s,z,0)\right)+\p F_0(s,x(s,z,0))
y_1(s,z) \Big)ds,\vspace{0.3cm}\\
y_i(t,z)=&i!\int_0^{t}\Big(F_i\left(s,x(s,z,0)\right)
+\sum_{l=1}^{i}\sum_{S_l} \dfrac{1}{b_1!\,b_2!2!^{b_2}\ldots
b_l!l!^{b_l}}\vspace{0.2cm}\\
&\cdot\p^L F_{i-l} \left(s,x(s,z,0)\right)
\bigodot_{j=1}^ly_j(s,z)^{b_j}\Big)ds, \text{ for }\, i=2,\ldots,k.
\end{array}
\end{equation}
See, for instance,
\cite{LliItiNov2015,LliNovTeiN2014,LliNovTeiN2014c}.
\end{remark}

For more details on the results of this  section see
\cite{LliNovCand}.

\subsection{Standard form and main result}
This section is devoted to introduce the class of non-smooth differential systems of our interest as well as our main result. 

First of all, we introduce the Filippov's  convention for a class of non-smooth differential systems. A \textit{piecewise smooth vector field} defined in an open bounded
set $U\subset\R^m$ is a function $X:U\rightarrow\R^n$ which is smooth
except on a set $\Sigma$ of zero measure, called the {\it
discontinuity set}. We assume that
$U\setminus\Sigma$ is a finite union of disjoint open sets $U_i,$
$i=1,2,\ldots,n,$ where the restriction $X_i=X\big |_{U_i}$ can be
extended continuously to $\ov{U_i}$. The local orbit of $X$ at a point
$p\in U_i$  is defined as usual for a differential system. The local orbit of $X$ at a point $p\in\Sigma$ follows the Filippov's convention \cite{Filippov}. In what follows, we shall state this convention for the so-called crossing points.

Assume that $\Sigma$ is locally a codimension one embedded submanifold of $\R^n$ around $p$. In this case, there exists a small neighborhood $U_p$ and a smooth function $h:U_p\rightarrow \R,$ for which $0$ is a regular value, such that $\Sigma\cap U_p=h^{-1}(0)$. Denote $U_p^+=\{z\in U_p:\, h(z)\geq 0\}$ and $U_p^-=\{z\in U_p:\, h(z)\leq 0\},$ and write $X^{\pm}=X\big|_{U_p^{\pm}}.$ We say that $p$ is a {\it crossing point} provided that $X^+ h(p) X^-h(p)>0,$ where $X^{\pm}h(p)=\langle\nabla h(p),X^{\pm}(p)\rangle.$ Denote by $\Sigma^c\subset \Sigma$ the set of crossing points, which is open in $\Sigma.$
 For a point $p\in\Sigma^c$, the local orbit
of $X$ at $p$ is given as the concatenation of the local trajectories
of $X^{\pm}$ at $p$. In this case, we say that the orbit {\it crosses}
the discontinuity set. In this paper, we shall only deal with crossing points. See \cite{Filippov,GST} for the Filippov's convention of local trajectories at points in $\Sigma \setminus \Sigma^c$.

The above convention was stated for vector fields, that is, autonomous differential systems. Nevertheless, it can also be stated for non-autonomous differential systems $x'=F(t,x;\e)$ just by considering the extended autonomous differential system $(\dot t, \dot{x})=\widetilde F(t,x;\e)=(1,F(t,x;\e))$, $(t,x)\in\R \times \R^n.$ 

In what follows, we define the class of nonsmooth differential systems of our interest. Let $n>1$ be a positive integer, $0<t_1<t_2<\ldots<t_{n-1}<T$ an ordered sequence of real numbers, and $D$ an open subset of $\R^m$. For $i=0,1,\ldots,k$ and $j=1,2,
\ldots,n,$ let $F_i^j: \R \times D \to \R^m$ and $R^j:
\R \times D \times (-\e_0, \e_0) \to \R^m$ be 
$C^{k+1}$ functions $T$-periodic in the variable $t$.  Define the following piecewise smooth functions
\begin{equation*}\label{funcF}
\begin{array}{l}
\displaystyle F_i(t,x)=\sum_{j=1}^{n}\chi_{[t_{j-1},t_j]}(t)
F_i^j(t,x) ,\,\, i=0,1,...,k,\quad \text{and}\vspace{0.1cm}\\
\displaystyle
R(t,x,\e)=\sum_{j=1}^{n}\chi_{[t_{j-1},t_j]}(t)R^j(t,x,\e),
\end{array}
\end{equation*}
where $\chi_A(t)$ is the characteristic function of $A$ defined as
\begin{equation*} \chi_A(t)=\begin{cases}
1 & \text{if $t \in A$,}\\
0 & \text{if $t \not\in A$.}\\
\end{cases}
\end{equation*}
Consider the $T$-periodic non-autonomous differential system
\begin{equation} \label{perturbedsystem} x'=F(t,x,\e)=\sum_{i=0}^k
\e^i F_i(t,x) + \e^{k+1}R(t,x,\e).
\end{equation}
The discontinuity set $\Sigma$ 
is given by
\[
\Sigma=(\{t=0 \equiv T\} \cup \{t=t_1\} \cup \ldots \cup
\{t=t_{n-1}\}) \cap \mathbb \R \times D,
\]
Following the convention above, one can easily see that $\Sigma=\Sigma^c.$ Indeed, for each $j=1,2,\ldots,n,$ and $t \in [t_{j-1},t_{j}],$ system \eqref{perturbedsystem} writes
\begin{equation}
\label{sector} x'=F_j(x,t;\e)=\sum_{i=0}^k\e^i F_i^j(t,x) + \e^{k+1}R^j(t,x,\e).
\end{equation}
For $(t_j,x)\in\Sigma$ we have that the connected component  of $\Sigma$ containing $(t_j,x)$ is given by $\Sigma_i=h_j^{-1}(0),$ where $h_j(t,x)=t-t_j$.  Considering the extended system $\widetilde F_j(t,x;\e)=(1,F_j(t,x;\e)),$ we see that $\widetilde F_{j} h_j(t_j,x;\e)\widetilde F_{j+1} h_j(t_j,x;\e)=1>0.$

For each $z \in D$ and $|\e|\neq0$ sufficiently
small, we denote by $x(\cdot,z,\e):[0,t_{(z,\e)}) \to \R^m$ the
solution of system \eqref{perturbedsystem} such that  $x(0,z,\e)=z$,
where $[0,t_{(z,\e)})$ is the interval of definition of
$t\mapsto x(t,z,\e)$. We shall assume  hypothesis $(H_b)$ for system \eqref{perturbedsystem}, that is, the solutions of the unperturbed system 
\begin{equation}
\label{unperturbedsystem} x'= F_0(t,x),
\end{equation}
starting at points of $\mathcal Z,$ defined in \eqref{zdefi}, are $T$-periodic.



In \cite{LliNovRod}, the averaging theory was developed for system \eqref{perturbedsystem} assuming  $\dim (\mathcal{Z})=m$. Here, we are interested in the case $\dim(\mathcal{Z})<m$. Accordingly, we shall prove that the bifurcation functions \eqref{bifurcationfunction} also provides sufficient conditions for the existence of periodic solutions of \eqref{perturbedsystem} bifurcating from $\mathcal{Z}$  for $|\e| \neq 0$ sufficiently small.

In what follows, we shall state our main result which basically says that 
the simple zeros of the bifurcation functions \eqref{bifurcationfunction} also control the bifurcation of isolated periodic solutions of the non-smooth system \eqref{perturbedsystem}. However, in order to obtain the averaged functions \eqref{averagefunctions} we need to compute the fundamental
matrix $Y(t,z)$ of the variational equation
\begin{equation}\label{totalvarsys}
y'=\frac{\p}{\p x}F_0 (t, x(t,z,0)) y.
\end{equation}
Notice that, for each
$j=1,2,\ldots,n$, if $x_j(t,z,\e)$ denotes the solution of
\eqref{sector} for $t_{j-1}\leq t \leq t_j$, then function $t \mapsto
(\p x_j/\p z)(t,z,0)$ is a solution of \eqref{totalvarsys} for
$t_{j-1}\leq t \leq t_j$. 
Recall that the right product of a solution of the variational equation
\eqref{totalvarsys} by constant matrix is still a solution. Therefore, a fundamental matrix 
$Y(t,z)$ of \eqref{totalvarsys} can be built as follows:
\begin{equation*} \label{relation}
Y(t,z)=\begin{cases}
Y_1(t,z) & \text{if $0=t_0\leq t \leq t_1$},\\
Y_2(t,z) & \text{if $t_1\leq t \leq t_2$},\\
\vdots \\
Y_n(t,z) & \text{if $t_{n-1}\leq t \leq t_n=T$},
\end{cases}
\end{equation*} with
\begin{equation}
\label{defined1}
\begin{array}{RL}
Y_1(t,z)=&\frac{\p x_1}{\p z}(t,z,0), \quad  \text{and}\\
Y_j(t,z)=&\frac{\p x_j}{\p z}(t,z,0) \left(\frac{\p x_j}{\p
z}(t_{j-1},z,0)\right)^{-1} Y_{j-1}(t_{j-1},z), \quad \text{for
$j=2,3,\ldots,n$}.
\end{array}
\end{equation}
Notice that $Y(0,z)=Id.$

Now, we are ready to state our main result.

\begin{mtheorem}\label{maintheorem}
Let $\Delta_{\al}$ denote the lower right corner $(m-d)\times (m-d)$
matrix of the matrix $Id-Y^{-1}(T,z_{\al})$. In addition to hypothesis $(H_{b}),$ we assume that 
\begin{itemize}
\item[(i)] for each $\al \in \overline{V}$, $\det \Delta_{\al} \neq 0,$
\item[(ii)] $f_1=f_2=\ldots=f_{k-1}=0$ and $f_k$ is not identically zero.
\end{itemize}
If there exists
$\al^{\ast} \in V$ such that $f_k(\al^{\ast})=0$ and 
$\det(Df_k(\al^{\ast}))\neq 0$, then there exists a $T$-periodic
solution $\f(t,\e)$ of \eqref{perturbedsystem} such that
$|\f(0,\e)-z_{\al^{\ast}}|=\CO(\e)$.
\end{mtheorem}

\begin{remark}
	
	The derivatives $\partial^j F_{i}(s,z)$, which appears in \eqref{bifurcationfunction}, are computed as follows:
	\[
	\frac{\partial^j F_{i}}{\p z}(s,z)
	=\sum_{j=1}^{n}\chi_{[t_{j-1},t_j]}(s)\frac{\partial^j F_{i}^j}{\p
		z^j}(s,z).
	\]
\end{remark}

This paper is organized as follows. In Section \ref{andronov}, we present a wide class of non-smooth differential systems for which Theorem \ref{maintheorem} can be applied.  In Section
\ref{theaveragefunctions}, we present the explicit formulae
\eqref{averagefunctions} for the average functions of the non-smooth
differential system \eqref{perturbedsystem}. Finally, in Section \ref{proof}
we prove Theorem \ref{maintheorem} and in Section \ref{examples} we
provide two applications of Theorem \ref{maintheorem}.

\section{Andronov-Hopf equilibria}
\label{andronov}
In this section, we provide a class of non-smooth vector fields that can be studied in light of the previous result, namely, non-smooth perturbations of Andronov-Hopf equilibria. 

An Andronov-Hopf equilibrium is characterized by a singularity admitting one pair of purely imaginary conjugate eigenvalues such that any other eigenvalue is real and nonvanishing.  The Andronov-Hopf bifurcation appears in the study of many differential models describing real phenomena, some of them are non-smooth (see, for instance, \cite{Akhmet,CG,LliNovZel,SM2,SM,EZR}  and the references quoted therein). 

In what follows, we describe this class of non-smooth differential system. Let $D =D_1\times D_2$ be a subset of $\R^{2+m}$, where $D_1 \subset \R^2$ and $D_2 \subset \R^m$. Take $n>1$ a positive integer, $\al_0=0$, $\al_n=2\pi$ and $\al=(\al_1,\ldots,\al_{n-1})\in\mathbb{T}^{n-1}$ a $(n-1)$-tuple of angles such that $0=\al_0<\al_1<\al_2<\cdots<\al_{n-1}< \al_n=2\pi$. Consider $\X(u,v,w;\e)=(X_1,X_2,\ldots,X_n)$ a $n$-tuple of smooth vector fields defined on an open bounded neighborhood $D \subset \R^{2+m}$ of the origin and depending on a small parameter $\e$.  For each $j=1,2, \ldots, n,$ let $C_j$ be a subset of $D$ such that the intersection of $C_j$ with the plane $w=0$ is a sector in the plane $uv$ that is delimited by rays starting at the origin and passing through point $(\cos \al_j, \sin \al_j)$.

Now let $Z_{\X,\al}:D\rightarrow\R^{2+m}$ be a piecewise smooth vector field defined as $Z_{\X,\al}(u,v,w;\e)$ $=X_j(u,v,w;\e)$ when $(u,v,w)\in C_j$, and consider the following discontinuous piecewise smooth differential system
\begin{equation}\label{piece}
(\dot u,\dot v, \dot w)^T=Z_{\X,\al}(u,v,w;\e).
\end{equation}
The above notation means that at each sector $C_j$ we are considering a smooth differential system of the form
\begin{equation}\label{piecesector}
(\dot u,\dot v, \dot w)^T=X_j(u,v,w;\e)=X_0(u,v,w)+ \displaystyle \sum_{i=1}^k \e^i X_{i}^j(u,v,w)+\e^{k+1}H^j(u,v,w,\e),
\end{equation}where $k>1$ is an integer.

As our main hypothesis regarding system \eqref{piece}, we assume that the origin is an equilibrium of  $Z_{\X,\al}(u,v,w;\e)$ and
 \[
 Z_{\X,\al}(u,v,w;0)=X_0(u,v,w)=\Big(-v+P(u,v,z)\,,\,u+Q(u,v,z)\,,\,\lambda w+R(u,v,w)\Big),
 \]
 where $P,Q,R$ vanish at the origin and $\lambda=(\tilde\lambda_1, \ldots, \tilde \lambda_m)^T$ is such that $\tilde\lambda_s$ is negative, for $s=1,2, \ldots, m.$
Notice that we are assuming that the \textit{unperturbed system} $Z_{\X,\al}(u,v,w;0)$  is smooth and admits an Andronov-Hopf equilibrium at the origin.

Now, changing $u \to \e U, v \to \e V$ and $w \to \e W$ we get, for each $j=1, \ldots, n,$ a differential system of the form
\begin{equation}
\label{shopf2}
\begin{cases}
\dot{U}=-V+\displaystyle \sum_{i=1}^k \e^i \,\mathcal{G}_{1,i}(U,V,W)+ \e^{k+1}\mathcal{R}_{1}(U,V,W, \e),\\
\dot{V}=U+\displaystyle \sum_{i=1}^k \e^i\, \mathcal{G}_{2,i}(U,V,W)+ \e^{k+1}\mathcal{R}_{2}(U,V,W, \e),\\
\dot{W}=\lambda W+\displaystyle\sum_{i=1}^k \e^i\, \mathcal{G}_{3,i}(U,V,W)+ \e^{k+1}\mathcal{R}_{3}(U,V,W, \e),
\end{cases}
\end{equation}

In order to use Theorem \ref{maintheorem} for studying system \eqref{piece}, it has to be written in its standard form. Doing the cylindrical change of variables $U=r\,\cos\T$,  $V=r\,\sin\T$, $W=w$ and assuming that each $X_j$ has the form \eqref{shopf2}, system \eqref{piecesector} becomes
\begin{equation*}
\begin{array}{rl}
&r'(\T)=\dfrac{\dot r(t)}{\dot \T(t)}=\displaystyle\sum_{i=0}^k\e^i F_{1,i}^j(\T,r,w) + \e^{k+1}R_1^j(\T,r,w,\e),\\
&w'(\T)=\dfrac{\dot w(t)}{\dot \T(t)}=\displaystyle\sum_{i=0}^k\e^i F_{2,i}^j(\T,r,w) + \e^{k+1}R_2^j(\T,r,w,\e).
\end{array}
\end{equation*}
Now for $\T \in [\al_{j-1},\al_j],$  $F_{i,1}^j, F_{i,2}^j :\mathbb{S}^1\times D\rightarrow\R^{m+1}$ and $R_1^j, R_2^j :\mathbb S^1\times D\times(-\e_0,\e_0)\rightarrow\R^{m+1}$ are $C^{k+1}$ functions, depending on the vector fields $X_i^j$, and  $2\pi$-periodic in the first variable, with $D$ being an open bounded set and $\mathbb{S}^1\equiv\R/(2\pi\Z)$. Denoting
\begin{equation*}
\begin{array}{l}
\displaystyle F_{1,i}(\T,r,w)=\sum_{j=1}^{n}\chi_{[\al_{j-1},\al_j]}(\T) F_{1,i}^j(\T,r,w) ,\,\, i=0,1,...,k,\quad \vspace{0.1cm}\\
\displaystyle
\displaystyle F_{2,i}(\T,r,w)=\sum_{j=1}^{n}\chi_{[\al_{j-1},\al_j]}(\T) F_{2,i}^j(\T,r,w) ,\,\, i=0,1,...,k,\vspace{0.1cm}\\
\displaystyle R_1(\T,r,w,\e)=\sum_{j=1}^{n}\chi_{[\al_{j-1},\al_j]}(\T)R_1^j(\T,r,w,\e),\quad \text{and}\vspace{0.1cm}\\
\displaystyle
R_2(\T,r,w,\e)=\sum_{j=1}^{n}\chi_{[\al_{j-1},\al_j]}(\T)R_2^j(\T,r,w,\e),
\end{array}
\end{equation*}
system \eqref{piece} writes
\begin{equation*}
\begin{split}
r'(\T)= \displaystyle \sum_{i=0}^k\e^i F_{1,i}(\T,r,w)+\e^{k+1} R_1(\T,r,w,\e),\\
w'(\T)=\displaystyle\sum_{i=0}^k\e^i F_{2,i}(\T,r,w)+\e^{k+1} R_2(\T,r,w,\e).
\end{split}
\end{equation*}
Notice that the unperturbed system is $(r'(\T),w'(\T))=(0,\lambda w),$ so its solution for an initial condition $(r_0,w_0^1, \ldots, w_0^m)$ is provided by $$(r_0, w_0^1  e^{\lambda_1}, \ldots,  w_0^m e^{\lambda_m}).$$
Considering the set $\mathcal{Z}=\{(r,0,\ldots,0); r>0\} \subset \R^{1+m},$ the solution of the unperturbed system is $2\pi$--periodic for every initial condition $(r_0,w_0^1, \ldots, w_0^m) \in \mathcal Z$. Moreover, fundamental matrix $Y$ is such that $Id-Y^{-1}$ satisfy that $\det(\Delta_r) \neq 0$ for points in $\mathcal Z$.
Hence, system \eqref{piece} satisfies all hypotheses of Theorem \ref{maintheorem} and, then, we can use our main result to study isolated periodic solutions for systems like \eqref{piece}.

\section{An algorithm for the bifurcation functions}\label{theaveragefunctions}

In this section, we shall provide an algorithm for computing the
bifurcation functions \eqref{bifurcationfunction} for the non-smooth case.
Their expressions are defined recurrently and uses \textit{Bell
polynomials}, instead of the set $S_l,$  which are already implemented in computer algebra systems such as Mathematica and Maple.  In this way,  the computation of the higher order bifurcation functions is significantly simplified. In \cite{N},
it was proven that the average functions defined for smooth cases can
be computed using \textit{Bell polynomials}. In \cite{LliNovRod},
the authors did the same for the non-smooth case. 

For each pair of
nonnegative integers $(p,q)$, the partial Bell polynomial is defined
as
\[
B_{p,q}(x_1,x_2,\ldots,x_{p-q+1})=\sum_{\widetilde{S}_{p,q}}\frac{p!}{b_1!
b_2! \ldots b_{p-q+1}!} \prod_{j=1}^{p-q+1}
\bigg(\frac{x_j}{j!}\bigg)^{b_j},
\]
where $\widetilde{S}_{p,q}$ is the set of all $(p-q+1)$-tuple of
nonnegative integers $(b_1,b_2, \ldots, b_{p-q+1})$ satisfying $b_1 +
2b_2 + \ldots + (p-q + 1)b_{p-q+1} = p$, and $b_1 + b_2 + \ldots +
b_{p-q+1} = q$. Moreover, if $g$ and $h$ are sufficiently smooth
functions, Fa\'{a} di Bruno's formula for the $L$-th derivative of a composite function gives
\begin{equation*}
\label{derivativeBell} \dfrac{d^l}{d x^l} g(h(x))=\sum_{m=1}^l
g^{(m)}(h(x)) B_{l,m}(h'(x),h''(x), \ldots, h^{(l-m+1)}(x)).
\end{equation*}

\subsection{Average Functions}

In this subsection, we develop a recurrence to compute the average
function \eqref{averagefunctions} in the particular case of the
non-smooth differential equation \eqref{perturbedsystem}. So,
consider the functions $w_i^j:(t_{j-1},t_{j}]\times D \to \R^m$ defined
recurrently for $i=1,2,\ldots,k$ and $j=1,2,\ldots,n,$ as
\begin{equation}\label{w11}
\begin{array}{RL}
w_1^1(t,z)=&\int_0^{t} \bigg( F_1^1(s,x(s,z,0)) + \partial
F_0^1(s,x(s,z,0))w_1^1(s,z) \bigg)ds,\vspace{0.3cm}\\
w_i^1(t,z)=& i!\int_{0}^{t}\bigg(F_i^{1}(s,x(s,z,0))+\vspace{0.2cm}\\
&\sum_{l=1}^{i}\sum_{S_l}\dfrac{1}{b_1!\,b_2!2!^{b_2}\ldots b_l!l!^{b_l}}
\cdot\partial ^L F_{i-l}^{1}(s,x(s,z,0)) \bigodot_{m=1}^l w_m^{1}(s,z)^{b_m}
\bigg)ds,\vspace{0.3cm}\\
w_i^j(t,z)=&w_i^{j-1}(t_{j-1},z)+i! \int_{t_{j-1}}^{t}\bigg(F_i^{j}(s,x(s,z,0))
+\vspace{0.2cm}\\
&\sum_{l=1}^{i}\sum_{S_l}\dfrac{1}{b_1!\,b_2!2!^{b_2}\ldots
b_l!l!^{b_l}}\cdot\partial ^L F_{i-l}^{j}(s,x(s,z,0))
\bigodot_{m=1}^l w_m^{j}(s,z)^{b_m} \bigg)ds.
\end{array}
\end{equation}
Since $F_0\neq0$ the recurrence defined in \eqref{w11} is an integral
equation and the next lemma solves it using Bell polynomials.

\begin{lemma}\label{lemmaexplicit}
For $i=1,2,\ldots,k$ and $j=1,2,\ldots,n$ the recurrence \eqref{w11}
can be written as follows
\begin{equation*}
\begin{array}{RL}
w_1^1(t,z)=&\!\!\!\!Y_1(t,z)\int_0^{t} Y_1^{-1}(s,z)F_1^1(s,x(s,z,0))ds,
\vspace{0.3cm}\\
w_1^j(t,z)=&\!\!\!\!Y_j(t,z)\bigg(Y_j^{-1}(t_{j-1},z)w_1^{j-1}(t_{j-1},z)
+\int_{t_{j-1}}^{t}Y_j^{-1}(s,z)F_1^{j}(s,x(s,z,0))ds\bigg),\vspace{0.3cm}\\
w_i^1(t,z)=&\!\!\!\!Y_1(t,z)\int_{0}^{t}Y_1^{-1}(s,z)\Big(i! F_i^1
\left(s,x(s,z,0)\right)\\
&+\sum_{m=2}^i  \p^m F_0^1(s,x(s,z,0)).B_{i,m}(w_1^1, \ldots,w_{i-m+1}^1),
\vspace{0.3cm}\\
&+\sum_{l=1}^{i-1}\sum_{m=1}^l \dfrac{i!}{l!} \p^m F_{i-l}^1(s,x(s,z,0)).
B_{l,m}(w_1^1,\ldots,w_{l-m+1}^1) \Big) ds, \vspace{0.3cm}\\
\end{array}
\end{equation*}

\begin{equation*}
\begin{array}{RL}
w_i^j(t,z)=&\!\!\!\!Y_j(t,z)\Big[Y_j^{-1}(t_{j-1},z)w_i^{j-1}(t_{j-1},z)+
\int_{t_{j-1}}^{t}Y_j^{-1}(s,z)\Big(i! F_i^j\left(s,x(s,z,0)\right)\\
&+\sum_{m=2}^i  \p^m F_0^j(s,x(s,z,0)).B_{i,m}(w_1^j, \ldots,w_{i-m+1}^j),
\vspace{0.3cm}\\
&+\sum_{l=1}^{i-1}\sum_{m=1}^l \dfrac{i!}{l!} \p^m
F_{i-l}^j(s,x(s,z,0)). B_{l,m}(w_1^j,\ldots,w_{l-m+1}^j) \Big)
ds.\Big]
\end{array}
\end{equation*}
\end{lemma}

\begin{proof}
The idea of the proof is to relate the integral equations \eqref{w11}
to the Cauchy problem and then solve it. For example, if $i=j=1$ the
integral equation is equivalent to the following Cauchy problem
\[
\dfrac{\p w_1^1}{\p t}(t,z)=F_1^1\left(t,x(t,z,0)\right)+\p
F_0^1\left(t,x(t,z,0)\right)w_1^1 \,\,\text{ with }\,\, w_1^1(0,z)=0,
\]
and solving this linear differential equation we get the expression
of $w_1^1(t,z)$ described in the statement of the lemma. For more
details see \cite{LliNovRod}.
\end{proof}

The next provides a formula for the average functions
\eqref{averagefunctions} for the class of non-smooth differential
systems studied in this paper.

\begin{proposition}
\label{equality} For $i=1,2,\ldots,k$, the average function
\eqref{averagefunctions} of order $i$ is \begin{equation*}
\label{naveregedfunction} g_i(z)=Y_n^{-1}(T,z) \frac{w_i^n(T,z)}{i!}.
\end{equation*}
\end{proposition}

\begin{proof}
For each $i=1,2,\ldots,k$ we define
\begin{equation*}\label{zprom}
w_i(t,z)=\sum_{j=1}^{n}\chi_{[t_{j-1},t_j]}(t)w_i^j(t,z).
\end{equation*}
Given $t \in [0,T]$ there exists a positive integer $\bar{k}$ such
that $t \in (t_{\bar{k}-1}, t_{\bar{k}}]$ and, therefore,
$w_i(t,z)=w_i^{\bar{k}}(t,z)$. By applying the proof of Proposition $2$ of
\cite{LliNovRod}, we obtain
\begin{equation}\label{expressionz1}
\begin{array}{RL}
w_1(t,z)=&\int_0^{t} \bigg( F_1(s,x(s,z,0)) + \partial F_0(s,x(s,z,0))w_1(s,z) \bigg)ds,
\vspace{0.3cm}\\
w_i(t,z)=&i!\int_{0}^{\T}\bigg(F_i(s,x(s,z,0))+\\&\sum_{l=1}^{i}\sum_{S_l}\dfrac{1}{b_1!\,b_2!2!^{b_2}\cdots b_l!l!^{b_l}}\cdot\partial ^L F_{i-l}(s,x(s,z,0)) \bigodot_{m=1}^l
w_m(s,z)^{b_m} \bigg)ds.
\end{array}
\end{equation}

Since according to Remark \ref{remark} we can consider functions \eqref{y}
as implicitly provided, we compute the derivatives in the variable $t$ of
functions \eqref{expressionz1} and \eqref{y1} for $i=1$, and we
see that the functions $w_1(t,z)$ and $y_1(t,z)$ satisfy the same
differential equation. Moreover, for each $i=2,\ldots,k$, the integral
equations \eqref{y1} and \eqref{expressionz1}, which provide
respectively $y_i$ and $w_i$, are defined by the same recurrence. Then,
functions $y_i$ and $w_i$ satisfy the same differential equations
for $i=1,2,\ldots,k$, and their initial conditions coincide. Indeed,
let $i \in \{1,2,\ldots,k\}$, since $y_i(0,z)=0$ and, by
\eqref{expressionz1}, $w_i(0,z)=0$, it follows that the initial
conditions are the same. Applying the Existence and Uniqueness
Theorem on the solutions of the differential system we get
$y_i(t,z)=w_i(t,z)$, for all $i \in \{1,2,\ldots,k\}$.
\end{proof}

\subsection{Bifurcation Functions}

In this subsection, we shall write the bifurcation functions
\eqref{bifurcationfunction} and the functions $\gamma_i(\al)$ provided
by \eqref{gammadefi} in terms of Bell polynomials.

\begin{claim}
The bifurcation function \eqref{bifurcationfunction} is provided by
\[f_i(\al)=\pi g_i(z_{\al}) +  \sum_{l=1}^i \sum_{m=1}^l
\dfrac{1}{l!} \p^m_b \pi g_{i-l}(z_\al)
B_{l,m}(\gamma_1(\al),\ldots,\gamma_{l-m+1}(\al)),\] where
\begin{equation*}
\begin{array}{RL}
\gamma_1(\alpha)=& -\Delta_{\al}^{-1} \pi^{\perp} g_1(z_{\al}) \quad \text{and}\\
\gamma_i(\alpha)=&-\Delta_{\al}^{-1}\bigg(\sum_{l=0}^{i-1} \dfrac{i!}{l!}
\sum_{m=1}^{l} \p_b^m \pi^{\perp}g_{i-l}(z_\al) B_{l,m}(\gamma_1(\al),\ldots,
\gamma_{l-m+1}(\al)) \\
&+\sum_{m=2}^{i} \p_b^m \pi^{\perp}g_{0}(z_\al)
B_{i,m}(\gamma_1(\al),\ldots,\gamma_{i-m+1}(\al)) \bigg).
\end{array}
\end{equation*}
\end{claim}

\begin{proof}
Expression \eqref{bifurcationfunction} was obtained in \cite{LliNovCand} by using the Fa\'{a} di Bruno's formula for the $L$-th derivative of a composite function. This claim follows by applying the version of Fa\'{a} di Bruno's  formula in terms of Bell polynomials (see \cite{LliNovCand, N}).
\end{proof}

\section{Proof of Theorem \ref{maintheorem}}\label{proof}
For $j=1,2,\ldots,n$ let $\xi_j(t,t_{j-1},z_0,\e)$ be the solution of the
differential system \eqref{sector} such that
$\xi_j(t_{j-1},t_{j-1},z_0,\e)=z_0.$  Then, we define the recurrence

\begin{equation*}
\begin{split}
&x_1(t,z,\e)=\xi_1(t,0,z,\e)\\
&x_j(t,z,\e)=\xi_j(t,t_{j-1},x_{j-1}(t_{j-1},z,\e),\e),\quad
j=2,\ldots,n.
\end{split}
\end{equation*}
Since we are working in the cross region, it is easy to see that, for
$|\e| \neq 0$ sufficiently small, each $x_j(t,z,\e)$ is defined for
every $t\in[t_{j-1},t_j]$. Therefore $x(\cdot,z,\e):[0,T] \to \R^m$ is
defined as
\begin{equation*}
\label{sol} x(t,z,\e)=
\begin{cases}
x_1(t,z,\e) & \text{if $0=t_0 \leq t \leq t_1$},\\
x_2(t,z,\e) & \text{if $t_1 \leq t \leq t_2$},\\
\vdots \\
x_j(t,z,\e) & \text{if $t_{j-1} \leq t \leq t_j$},\\
\vdots \\
x_n(t, z,\e) & \text{if $t_{n-1} \leq t \leq t_n=T$}.
\end{cases}
\end{equation*}
Notice that $x(t,z,\e)$ is the solution of the differential equation
\eqref{unperturbedsystem} such that $x(0,z,\e)=z$. Moreover, the
following equality hold
\begin{equation*}\label{initial}
x_j(t_{j-1},z,\e)=x_{j-1}(t_{j-1},z,\e),
\end{equation*}
for $j=1,2,\ldots, n$.

The next lemma expands the solution $x_j(\cdot,z,\e)$ around $\e=0$.

\begin{lemma}\label{lemma}
For $j \in \{1,2,\ldots,n\}$ and $t_{z}^{j}>t_j,$ let
$x_j(\cdot,z,\e):[t_{j-1},t_{j})$ be the solution of \eqref{sector}.
Then
\[
x_j(t,z,\e)=x_j(t,z,0)+\sum_{i=1}^k \frac{\e^i}{i!} w_i^j(t,z) +
\CO(\e^{k+1}).
\]
\end{lemma}

\begin{proof}
First, having fixed $j\in\{1,2,\ldots,n\}$, we use the continuity of solution $x_j(t,z,\e)$ and the compactness of set
$[t_{j-1},t_{j}]\times\overline{D}\times[-\e_0,\e_0]$ to arrive at
\[
\int_{t_{j-1}}^{t} R^j(s,x_j(s,z,\e),\e)ds= \CO(\e), \quad t\in
[t_{j-1},t_{j}].
\]
Thus, integrating the differential equation \eqref{sector} from
$t_{j-1}$ to $t$, we get
\begin{equation}\label{fundlemma}
\begin{array}{RL}
x_j(t,z,\e) &= x_j(t_{j-1},z,\e) + \sum_{i=0}^k \e^i \int_{t_{j-1}}^{t}
F_i^j(s,x_j(s,z,\e)) ds + \CO(\e^{k+1}), \quad \text{and}\\
x_j(t,z,0) &= x_j(t_{j-1},z,0) + \int_{t_{j-1}}^{t}
F_0^j(s,x_j(s,z,0)) ds.
\end{array}
\end{equation}

Because of the differentiable dependence of the solutions of a differential
system on its parameters, the function $\e \mapsto x_j(t,z,\e)$ is a
$C^{k+1}$ map. The next step is to compute the Taylor expansion
of $F_i^j(t,x_j(t,z,\e))$ around $\e=0$, and for this we use the
Fa\'{a} di Bruno's Formula about the $l$-th derivative of a composite
function, which guarantees that, if $g$ and $h$ are sufficiently smooth
functions, then
\[
\dfrac{d^l}{d\al^l}g(h(\al))= \sum_{S_l}
\dfrac{l!}{b_1!\,b_2!2!^{b_2}\ldots b_l!l!^{b_l}}g^{(L)}
(h(\al))\bigodot_{j=1}^l\left(h^{(j)}(\al)\right)^{b_j},
\]
where $S_l$ is the set of all $l$-tuples of non-negative integers
$(b_1,b_2,\ldots,b_l)$ satisfying $b_1+2b_2+\ldots+lb_l=l$, and
$L=b_1+b_2+\ldots+b_l$.

For each $i=0,1,...,k-1,$ expanding $F_i^j(s,x_j(s,z,\e))$ around $\e=0$ we get
\begin{equation}\label{s10}
\begin{array}{RL}
F_i^j(s,x_j(s,z,\e))  =&  F_i^j(s,x_j(s,z,0))+\\
&\sum_{l=1}^{k-i}  \sum_{S_l}\dfrac{\e^l}{b_1!\,b_2!2!^{b_2}\ldots
b_l!l!^{b_l}} \partial ^L F_i^j(s,x_j(s,z,0)) \bigodot_{m=1}^l
r_m^j(s,z)^{b_m},
\end{array}
\end{equation}
where
\[
r_m^j(s,z)=\frac{\p^m}{\p \e^m}x_j(s,z,\e)\Big|_{\e=0},
\]
and for $i=k$
\begin{equation}\label{s11}
F_k^j(s,x_j(s,z,\e))= F_k^j(s,x_j(s,z,0)) + \CO(\e).
\end{equation}

Substituting \eqref{s10} and \eqref{s11} in \eqref{fundlemma} we get
\begin{equation*}\label{delta}
\begin{array}{RL}
x_j(t,z,\e)=
&x_{j}(t_{j-1},z,\e) + \int_{t_{j-1}}^t\Bigg( \sum_{i=0}^k \e^i
F_i^j(s,x_j(s,z,0)) ds\\
&+\sum_{i=0}^{k-1} \sum_{l=1}^{k-i} \e^{l+i}
\sum_{S_l}\dfrac{1}{b_1!\,b_2!2!^{b_2}\ldots
b_l!l!^{b_l}}\\
&\cdot \partial ^L F_i^j(s,x_j(s,z,0)) \bigodot_{m=1}^l
r_m^j(s,z)^{b_m}\Bigg)ds+ \CO(\e^{k+1}).
\end{array}
\end{equation*}
Then, the proof of the lemma ends using the next two claims.

\begin{claim}\label{claim1}
For $j=1,2,\ldots,n$ we have
\[
x_j(t,z,\e)=x_j(t,z,0)+\sum_{i=1}^k \frac{\e^i}{i!} r_i^j(t,z) +
\CO(\e^{k+1}).
\]
\end{claim}

\begin{claim}
	\label{claim2}
The equality  $r_i^j=w_i^j$ holds for $i=1,2,\ldots,k$ and
$j=1,2,\ldots,n.$
\end{claim}
The proof of Claims \ref{claim1} and \ref{claim2} is achievable by following the steps described in the proof of Claims 1 and 2 of \cite{LliNovRod}, respectively.
\end{proof}

\begin{proof}[Proof of Theorem \ref{maintheorem}]
Consider the \textit{displacement function}
\begin{equation}
\label{displacement} h(z,\e)=x(T,z,\e)-z=x_n(T,z,\e)-z
\end{equation}

It is easy to see that $x(\cdot,\overline{z},\overline{\e})$ is a
$T$-periodic solution if and only if
$h(\overline{z},\overline{\e})=0$. Moreover, to study the zeros of
\eqref{displacement} is equivalent to study the zeros of
\begin{equation}
\label{functiong} g(z,\e)=Y_n^{-1}(T,z) h(z,\e).
\end{equation}

Based on Lemma \ref{lemma}, we have
\begin{equation}
\label{xn} x_n(T,z,\e)=x_n(T,z,0)+\sum_{i=1}^k \frac{\e^i}{i!}
w_i^n(T,z) + \CO(\e^{k+1}),
\end{equation}
for all $(t,z) \in \mathbb{S}^1 \times D$. By replacing \eqref{xn} in
\eqref{functiong}, it follows that
\begin{equation}
\label{functiongg}
\begin{array}{RL}
g(z,\e)=&Y_n^{-1}(T,z)\left(x_n(T,z,0)-z+\sum_{i=1}^k \frac{\e^i}{i!}
w_i^n(T,z) + \CO(\e^{k+1})\right)\\
=&Y_n^{-1}(T,z)(x_n(T,z,0)-z) +\sum_{i=1}^k g_i(z) + \CO(\e^{k+1})\\
=&\sum_{i=0}^k g_i(z) + \CO(\e^{k+1}),
\end{array}
\end{equation}
where $g_0(z)=Y_n^{-1}(T,z)(x_n(T,z,0)-z)$.

From hypothesis \text{$(H_{a}$)} the function $g_0(z)$ vanishes on the submanifold
$\mathcal{Z}$, therefore hypothesis \text{$(H_{a}$)}  holds for function \eqref{functiongg}. In order to take the derivative of
$g_0(z)$ with respect to the variable $z$, we have the next claim.

\begin{claim}\label{lastclaim} For every $j \in \{1,2,\ldots,n\}$
\[Y_j(t_j,z)=\dfrac{\p x_j}{\p z}(t_j,z,0).\]
\end{claim}

The proof will be done by induction on $j$. For $j=1$ the claim is
precisely the definition.  Assume that the claim is valid for
$j=j_0-1$ and we shall prove it for $j=j_0$. Since
$x_j(t_{j-1},z,\e)=x_{j-1}(t_{j-1},z,\e)$ for all $j=1,2,\ldots,n$ we
have
\begin{equation*}
\begin{array}{cl}
Y_{j_0}(t_{j_0},z)&=\dfrac{\p x_{j_0}}{\p
z}(t_{j_0},z,0)\left(\dfrac{\p x_{j_0}}{\p z}(t_{j_0-1},z,0)
\right)^{-1}Y_{j_0-1}(t_{j_0-1},z)\\
&=\dfrac{\p x_{j_0}}{\p z}(t_{j_0},z,0)\left(\dfrac{\p x_{j_0-1}}{\p z}
(t_{j_0-1},z,0)\right)^{-1}\dfrac{\p x_{j_0-1}}{\p z}(t_{j_0-1},z,0)\\
&=\dfrac{\p x_{j_0}}{\p z}(t_{j_0},z,0).
\end{array}
\end{equation*}
Hence, if $z \in \mathcal{Z}$ then
\begin{equation*}
\begin{array}{RL}
\frac{\p g_0}{\p z}(z)= &Y^{-1}(T,z) \left(\frac{\p x}{\p z}(T,z,0)-Id\right)\\
=&Y^{-1}(T,z)(Y(T,z)-Id)\\
=&Id-Y^{-1}(T,z),
\end{array}
\end{equation*}
which by assumption has as its lower right corner $(m-d) \times (m-d)$
matrix $\Delta_{\al}$ nonsingular. From here, the result follows from Proposition \ref{equality} and
Theorem \ref{LStheorem}. \end{proof}

\section{Examples}\label{examples}

This section is devoted to presenting some applications of Theorem \ref{maintheorem}. The first one is a 3D piecewise smooth system for which the plane $y=0$ is the switching manifold and admits a surface $z=f(x,y)$ foliated by periodic solutions. The second one is a 3D piecewise smooth system for which the algebraic variety $xy=0$ is the discontinuity set and the plane $z=0$ has a piecewise constant center. For these systems, we compute some of the bifurcation functions in order to study the persistence of periodic solutions. 

\subsection{non-smooth perturbation of a 3D system}

Let $f:\R^2 \to \R$ and $g:\R^2 \to \R$ be differential functions
such that $g(x,y)=f(x,y)+x \p_yf(x,y)-y\p_xf(x,y)$. Consider the
non-smooth vector field 
\begin{equation}\label{X1X2}
X_{\e}(x,y,z)=\left\{
\begin{array}{ll}
X^+_{\e}(x,y,z),&y>0\\
X^-_{\e}(x,y,z),&y<0
\end{array}\right.
\end{equation}
where
\[
\begin{array}{ll}
X^+_{\e}(x,y,z)=\left(-y +  \e (a_0+a_1 z)+\e^2 (a_2+a_3 z), \,\,
x ,\,\,-z+ g(x,y)\right),\quad\text{and}\\
X^-_{\e}(x,y,z)=\left(-y,\,\, x +  \e b_1 z+\e^2 (b_2+b_3) z,\,\,-z+
g(x,y)\right),
\end{array}
\]
with $a_0,\, a_1, \, a_2, \, b_1, \, b_2, \, b_3 \in \R$. Denote the discontinuity set by $\Sigma=\{(x,y,z) \in \R^3:\, y=0\}$.

Notice that surface $z=g(x,y)$ is an invariant set of the unperturbed vector field $X_0$. Indeed, considering function $\hat f(x,y,z)=z-f(x,y)$, we get
 \[ 
\langle\nabla \hat f(x,y,z),X_0(x,y,z) \rangle\big|_{z=f(x,y)}=0.
\]
Moreover, since $X_0(x,y,f(x,y))=\big(-y,x,x \p_yf(x,y)-y\p_xf(x,y)\big)$, we conclude that the invariant set $z=f(x,y)$ is foliate by periodic solutions.

Next result gives suficient conditions in order to guarantee the persistence of a periodic solution. Consider the function
\begin{equation}
\label{bifur1} f_1(r)=a_1\displaystyle\int_0^{\pi} f(r \cos \phi, r
\sin \phi) \cos \phi d\phi+ b_1 \displaystyle\int_{\pi}^{2\pi} f(r
\cos \phi, r \sin \phi) \sin \phi d\phi.
\end{equation}

\begin{theorem}\label{Texf1}
Consider the piecewise vector field \eqref{X1X2}. Then, for each $r*>0$, such that $f_1(r^*)=0$ and
$f_1'(r^*)\neq0$, there exists a crossing limit cycle $\f(t,\e)$ of
$X$ of period $T_{\e}=2\pi+\CO(\e)$ such that 
$\f(t,\e)=(x^*,y^*,f(x^*,y^*))+\CO(\e)$ with $|(x^*,z^*)|=r^*$.
\end{theorem}

In order to apply Theorem \ref{maintheorem} for proving Theorem \ref{Texf1}, we need to write system
\eqref{X1X2} in its standard form. Considering
cylindrical coordinates $x=r \cos \T$, $y=r \sin \T$, $z=z,$ the set
of discontinuity becomes $\Sigma= \{\T=0\} \cup \{\T=t_1\}$ with
$t_0=0, t_1=\pi$ and $t_2=2\pi$. The differential system $(\dot
x,\dot y, \dot z)=X_{\e}^+(x,y,z)$ in cylindrical coordinates writes
\begin{equation*}\label{polar1}
\begin{split}
& r'(t)=\e (a_0+a_1 z) \cos \T+\e^2( a_2+a_3  z) \cos \T,\\
& z'(t)=g(r \cos \T,r \sin \T)-z,\\
& \T'(t)=1-\e \dfrac{(a_0+a_1 z)   \sin \T}{r}-\e^2\dfrac{( a_2+a_3
z)  \sin \T}{r},
\end{split}
\end{equation*}
and the differential system $(\dot x,\dot y, \dot z)=X_{\e}^-(x,y,z)$ becomes
\begin{equation}\label{polar2}
\begin{split}
& r'(t)= \e b_1 z \sin \T+\e^2 (b_2+b_3 z) \sin \T,\\
& z'(t)=g(r \cos \T,r \sin \T)-z,\\
& \T'(t)=1+\e\frac{b_1  z \cos \T}{r}+\e^2\frac{( a_2+a_3  z) \cos
\T}{r}.
\end{split}
\end{equation}

Notice that, for each $j=1,2$ and $t_{j-1}\leq\T\leq t_j$, we have
$\dot \T(t)\neq0$ for $|\e|\neq0$ sufficiently small. Thus, in a
sufficiently small neighborhood of the origin we can take $\T$ as the
new independent time variable. Accordingly, system \eqref{polar2} becomes
\begin{equation*}\label{eqcil}
\begin{array}{cl}
&\dot r(\T)=\dfrac{r'(t)}{\T'(t)}=F_{01}(\T,r,z)+\e F_{11}(\T,r,z)+\e^2 F_{2}(\T,r,z)+\CO_1(\e^3),\\
&\dot z(\T)=\dfrac{z'(t)}{\T'(t)}=F_{02}(\T,r,z)+\e F_{12}(\T,r,z))+\e^2
F_{22}(\T,r,z)+\CO_2(\e^3).
\end{array}
\end{equation*}
Considering the notation of Theorem \ref{maintheorem}  we have
$F_i(\T,r,z)=\left(F_{i1}(\T,r,z),F_{i2}(\T,r,z)\right)$ for each $i
\in \{1,2\}$. Moreover, for each $i \in \{1,2\}$ the function
$F_i(\T,r,z)$ is written in the form
$F_i(\T,r,z)=\sum_{j=1}^{2}\chi_{[t_{j-1},t_j]}(\T) F_i^j(\T,r,z)$.

Defining $\tilde{f}(\T,r)=f(r \cos \T, r \sin \T)$ and
$\tilde{g}(\T,r)=g(r \cos \T, r \sin \T)$ we write explicitly the
expressions of $F_0, F_1^j$ and $F_2^j$ for $j \in \{1,2\}$,
\begin{equation*}\begin{array}{rl}
F_0(\T,r,z)=&(0,\,\, \tilde{g}(\T,r)-z),\\
F_1^1(\T,r,z)=&\left((a_0+a_1 z) \cos \T, \, \dfrac{(a_0+a_1 z)\sin
\T}{r}( \tilde{g}(\T,r)-z)\right),\\
F_1^2(\T,r,z)=&\left(b_1 z \sin \T, \, -\dfrac{b_1 z\cos \T}{r}(
\tilde{g}(\T,r)-z)\right),\\
F_2^1(\T,r,z)=&\Big((a_2 +a_3z) \cos \T+\dfrac{(a_0+a_1z)^2 \sin \T
\cos \T}{r}, \, \dfrac{ \sin \T}{r^2}\left((a_0+a_1z)^2  \sin
\T \right.\\
&\left. +(a_2+a_3z) r \right) (\tilde{g}(\T,r)-z)\Big) ,\\
F_2^2(\T,r,z)=&\left((b_2+b_3 z) \sin \T-\dfrac{b_1^2 z^2 \sin \T
\cos \T}{r}, \, \dfrac{ \cos \T}{r^2} \left(b_1^2 z \cos
\T-(b_2+b_3z) r\right) (\tilde{g}(\T,r)-z)\right).
\end{array}
\end{equation*}

The unperturbed system is smooth and its solution
$(r(\T,r_0,z_0),z(\T,r_0,z_0))$ with initial condition $(r_0,z_0)$ is provided by
\begin{equation}\label{solutionunperturbed}
r(\T)=\ov r(\T,r_0,z_0)=r_0,\quad z(\T)=\ov z(\T,r_0,z_0)= e^{-\T}
\left(z_0+\int_0^\T e^s \tilde{g}(s,r_0) ds \right).
\end{equation}
Consequently, a fundamental matrix solution of  \eqref{totalvarsys} is provided by 
\[
Y(\T,r_0,z_0)=\dfrac{\p(\ov r,\ov z)}{\p(r_0,z_0)}(\T,r_0,z_0)=\left(
\begin{array}{cc}
1 & 0 \\
G(\T,r_0) & e^{-\T} \\
\end{array}
\right),\]
where $G(\T,r_0)$ is the derivative of
$\ov z(\T,r_0,z_0)$ with respect to the variable $r_0$. Notice that, from \eqref{solutionunperturbed}, $G(\T,r_0)$ does not depend on $z_0$.

Let $\e_0>0$ be a real positive number and consider the set
$\mathcal{Z} \subset \R^2$ such that $\mathcal{Z}
=\{(r,\tilde{f}(0,r)) : r>\e_0 \}$. Notice that for $(r_0,z_0)=(r_0,\tilde f(0,r_0)) \in
\mathcal{Z}$ we have $z(\T,r_0,z_0)=\tilde{f}(\T,r_0)=f(r_0 \cos \T, r_0
\sin \T)$. Indeed, let $w(\T)=f(r_0 \cos \T, r_0 \sin \T)$. Thus
\begin{equation*}
\begin{array}{cl}
w'(\T)&=\p_x f(r_0 \cos \T, r_0 \sin \T)(-r_0\sin \T)+\p_y f(r_0 \cos \T,
r_0 \sin \T)(r_0\cos \T)\\
&=g(r_0 \cos \T, r_0 \sin \T)-f(r_0 \cos \T, r_0 \sin \T)\\
&=g(r_0 \cos \T, r_0 \sin \T)-w(\T)\\
&=\tilde{g}(\T,r_0)-w(\T).
\end{array}
\end{equation*}
The second equality holds because $g(x,y)=f(x,y)+x
\p_yf(x,y)-y\p_xf(x,y)$.  
Hence, for $(r_0,z_0) \in
\mathcal{Z}$ the solution $z(\T,r_0,z_0)$ is
$2\pi$-periodic. Moreover,
$$Id-Y^{-1}(2\pi,r,z)=\left(
\begin{array}{cc}
0 & 0 \\
\star & 1-e^{2\pi} \\
\end{array}
\right).$$
Consequently, $\Delta_{\al}=1-e^{2\pi} \neq 0$. Accordingly,
all the hypotheses of Theorem \ref{maintheorem} are satisfied.

\begin{proof}[Proof of Theorem \ref{Texf1}]
Denote by $(r,z_r)$ a point in $\mathcal{Z}$, that is $z_r=\tilde f(0,r)$. Notice that the
bifurcation function of first order is $f_1(r)=\pi g_1(r,z_r)$, where
$g_1$ is defined in \eqref{averagefunctions}. Indeed, from
definition $f_1(r)= \pi g_1(r,z_r) + \dfrac{\p \pi g_0}{\p b}
(r,z_r) \gamma_1(r)$. However,\[
g_0(r,z)=Y^{-1}(2\pi,r,z)((r,z(2\pi,r,z)))-(r,z(0,r,z)))=(0,\star),\]and
then $\pi g_0 \equiv 0$. Moreover,

\begin{equation*}\label{w11geral}
\begin{array}{l}
w_1^1(\T,r,z)=\Bigg(a_0 \sin \T+a_1\displaystyle\int_0^{\T} z(\phi) \cos
\phi d \phi, \,\,  G(\T,r)\left(a_0 \sin \T+a_1
\displaystyle\int_0^{\T} z(\phi) \cos \phi d \phi\right)- \\
\quad e^{-\T} \displaystyle\int_0^{\T} \left(e^{\phi} G(\phi,r)
(a_0+a_1 z(\phi))\cos \phi  + \sin \phi\dfrac{e^{\phi}(\tilde{g}
(\phi,r)-z(\phi))(a_0+a_1z(\phi))}{r}\right) d \phi\Bigg),
\end{array}
\end{equation*}
\begin{equation*}\label{w12geral}
\begin{split}
w_1^2(\T,r,z)=&Y(\T,r,z)
\Bigg[Y^{-1}(\pi,r,z)w_1^1(\pi,r,z)+\displaystyle\int_{\pi}^{\T}Y^{-1}(\phi,r,z)
F_1^2(\phi,r(\phi),z(\phi)) d \phi\Bigg]\\
=&Y(\T,r,z)\Bigg(a_1\displaystyle\int_0^{\pi} z(\phi) \cos \phi d
\phi+b_1\displaystyle\int_{\pi}^{\T} z(\phi) \sin \phi d \phi,\\
&\displaystyle\int_0^{\pi } \frac{e^{\phi } ((a_0+a_1 z(\phi )) (\sin
\phi (g(r \cos \phi ,r \sin \phi )-z(\phi ))-r \cos \phi \, G(\phi
,r))}{r} \, d\phi\\
&+\displaystyle\int_{\pi }^{\theta } -\frac{b_1 e^{\phi } z(\phi )
(\cos \phi  (g(r \cos \phi ,r \sin \phi )-z(\phi ))+r \sin \phi
\,G(\phi ,r))}{r} \, d\phi\Bigg).
\end{split}
\end{equation*}
Since $g_1(r,z)=Y^{-1}(2\pi,r,z)w_1^2(2\pi,r,z)$ and $f_1(r)=\pi
g_1(r,z_r)$ it follows that
\begin{equation}\label{bifu1}
f_1(r)=a_1\displaystyle\int_0^{\pi} f(r \cos \phi, r \sin \phi) \cos
\phi d\phi+ b_1 \displaystyle\int_{\pi}^{2\pi} f(r \cos \phi, r \sin
\phi) \sin \phi d\phi.
\end{equation}
So, based on Theorem \ref{maintheorem}, each positive  simple
zero of \eqref{bifur1} provides an isolated periodic solution
of system \eqref{X1X2}. This concludes this proof.
\end{proof}

The next result is an application of Theorem \ref{Texf1}. We shall use in its statement the concept of Bessel functions, which are defined as the canonical solutions $y(x)$ of Bessel's differential equation 
\[
x^2 \dfrac{d^2y}{dx^2}+x
\dfrac{dy}{dx}+(x^2-\alpha^2)y=0,\quad \alpha \in \mathbb{C}.
\]
This equation has two linearly independent solutions. Using
Frobenius' method we obtain one of these solutions, which is called a
\textit{Bessel function of the first kind}, and is denoted by
$J_{\alpha}(x)$. More details about this function can be found in
\cite{bessel}.

\begin{corollary}
Consider the piecewise vector field \eqref{X1X2}. 
\begin{itemize}
\item[(a)] If $f(x,y)=\cos x$, then the piecewise smooth vector field
$X$ admits a sequence of limit cycles $\f_i(t,\e)$ of $X$ of period
$T_{\e}$ such that $T_{\e}=2\pi+\CO(\e)$,
$\f_n(t,\e)=(x_n^*,y_n^*,\cos(x_n^*))+\CO(\e)$, and $|(x_n^*,z_n^*)|=n\pi/2$.

\item[(b)] If $f(x,y)=\sin x$, then the piecewise smooth vector field
$X$ admits a sequence of limit cycles $\f_i(t,\e)$ of $X$ of period
$T_{\e}$ such that $T_{\e}=2\pi+\CO(\e)$, 
$\f_i(t,\e)=(x_n^*,y_n^*,\sin(x_n^*))+\CO(\e)$, and $|(x_n^*,z_n^*)|=r_n^*$,
where each $r_n$ is a zero of the Bessel Function of First Kind,
$J_1(r)$.
\end{itemize}
\end{corollary}

\begin{proof}
For $f(x,y)=\cos x$, the bifurcation function \eqref{bifu1} reads
$f_1(r)=-(2b_1\sin r)/r$, and for $f(x,y)=\cos(x)$, the bifurcation
function \eqref{bifu1} reads $f_1(r)=a_1\pi J_1(r)$. Therefore, the
result follows directly from Theorem \ref{Texf1}.
\end{proof}

Notice that Theorem \ref{Texf1} cannot be applied when $f_1$ is
identically zero, which is the case when $f(x,y)=2x^2-y^2$, for instance.  For such cases we define the function

\begin{equation}\label{bifu2}
\begin{split}
f_2(r)&=\displaystyle \int_0^{\pi} \Bigg( a_1 \cos s \bigg(G(s,r)
\displaystyle \int_0^{s}\cos \phi(a_0+a_1 \tilde{f}(\phi, r))d \phi\\
&-e^{-s} \displaystyle \int_0^{s} e^{\phi} (a_0+a_1 \tilde{f}(\phi,
r))(r \cos \phi\,G(\phi,r)+(\tilde{f}(\phi, r)-\tilde{g}(\phi, r))
d\phi\\
&+a_2+a_3 \tilde{f}(\phi, r)+\frac{\sin s}{r} (a_0+a_1 \tilde{f}(s,
r))^2 \bigg)\Bigg) ds\\
&+\frac{e^{-2\pi}(1+e^{\pi})}{2(1-e^{2\pi})}(a_1e^{\pi}-b_1)
\bigg[\displaystyle \int_{0}^{\pi} e^{\phi}G(\phi,r) \cos
\phi(a_0+a_1 \tilde{f}(\phi, r))d\phi\\
&+\displaystyle \int_{0}^{\pi}\frac{e^{\phi}\sin \phi}{r}(a_0+a_1
\tilde{f}(\phi, r))(\tilde{g}(\phi, r)-\tilde{f}(\phi, r)) d\phi\\
&+b_1   \displaystyle\int_{\pi}^{2\pi} e^{\phi} G(\phi,r) \sin
\phi\tilde{f}(\phi, r) d \phi+\frac{b_1}{r} \displaystyle
\int_{\pi}^{2\pi} e^{\phi} \cos \phi(\tilde{g}(\phi,
r)-\tilde{f}(\phi, r)) d \phi\bigg]\\
&+\displaystyle \int_{\pi}^{2\pi}\Bigg( \frac{2}{r}(-b_1^2 \cos s
(\tilde{f}(s, r))^2+ \sin s (b_2+b_3\tilde{f}(s, r)) )\\
&+2b_1 \sin s\bigg( G(s,r) \displaystyle \int_{0}^{\pi}  \cos \phi
(a_0+a_1 \tilde{f}(\phi, r))  +b_1G(s,r) \displaystyle
\int_{\pi}^{s} \sin \phi \tilde{f}(\phi, r) d\phi\\
&+e^{-s} \bigg(\displaystyle \int_{0}^{\pi} -e^{\phi}\cos \phi\,
G(\phi,r)(a_0+a_1 \tilde{f}(\phi, r))+\frac{e^{\phi} \sin \phi}{r}
(\tilde{g}(\phi, r)-\tilde{f}(\phi, r))d \phi\\ &+b_1 \displaystyle
\int_{\pi}^{s} e^{\phi} \left(\frac{\cos {\phi}}{r} (\tilde{f}(\phi,
r)-\tilde{g}(\phi, r))-G(\phi,r) \sin \phi \right)d\phi\bigg)
\bigg)\Bigg)ds.
\end{split}
\end{equation}

\begin{theorem}\label{Texf2} Consider the piecewise vector field \eqref{X1X2}.  Assume
that $f_1\equiv 0$. Then, for each $r*>0$, such that $f_2(r^*)=0$ and
$f_2'(r^*)\neq0$, there exists a crossing limit cycle $\f(t,\e)$ of
$X$ of period $T_{\e}$ such that $T_{\e}=2\pi+\CO(\e)$,
$\f(t,\e)=(x^*,y^*,f(x^*,y^*))+\CO(\e)$, and $|(x^*,z^*)|=r^*$.
\end{theorem}

\begin{proof}
As we saw before $\pi g_0 \equiv 0$. So, from \eqref{bifurcationfunction}, we compute the bifurcation function of order $2$ as
\begin{equation}\label{ff2}
f_2(r)=\dfrac{\p \pi g_1}{\p b}(r,z_r) \gamma_1(r)+\pi g_2(r,z_r),
\end{equation}
where $\gamma_1(r)=-\dfrac{1}{1-e^{2 \pi }} \pi^{\perp}g_1(r,z_r)$
and
\begin{equation*}\label{projectionofg1}
\begin{split}
\pi^{\perp}g_1(r,z_r)&= \displaystyle\int_0^{\pi } \frac{e^{\phi }
((a_0+a_1 \tilde{f}(\phi, r)) (\sin \phi (g(r \cos \phi ,r \sin \phi
)-\tilde{f}(\phi, r))-r \cos \phi  G(\phi ,r,z))}{r} \, d\phi\\
&-b_1\displaystyle\int_{\pi }^{2\pi} \frac{ e^{\phi } \tilde{f}(\phi,
r) (\cos \phi  (g(r \cos \phi ,r \sin \phi )-\tilde{f}(\phi, r))+r
\sin \phi  G(\phi ,r,z))}{r} \, d\phi.
\end{split}
\end{equation*}
From Proposition \ref{equality}, we have $g_2(r,z_r)=Y^{-1}(2\pi,r,z)w_2^2(2\pi,r,z)/2$, where 
$w_i^j(2\pi,r,z)$ is provided in Lemma \ref{lemmaexplicit}. All these functions may be computed to get \eqref{ff2} as  \eqref{bifu2}. Again, in accordance with Theorem \ref{maintheorem}, each positive  simple
zero of \eqref{bifu2} provides an isolated periodic solution
of system \eqref{X1X2}. This concludes this proof.
\end{proof}

The next result is an application of Theorem \ref{Texf2}.

\begin{corollary}\label{corollary9}
Consider the piecewise vector field \eqref{X1X2} and let
$f(x,y)=2 x^2-y^2$. Assuming $a_1^2+b_1^2 \neq 0$ define
\begin{equation}\label{A0A1}
\begin{array}{l}
A_0=\dfrac{-80 b_2 (1-e^{\pi})}{(1+e^{\pi})4(15 a_1
b_1-b_1^2-14a_1^2)-5\pi(1-e^{\pi})(b_1^1+10a_1^2)},\vspace{0.2cm}\\
A_1=\dfrac{40
a_0((1+e^{\pi})(b_1-a1)-a1 \pi(1-e^{\pi})}{(1+e^{\pi})4(15 a_1
b_1-b_1^2-14a_1^2)-5\pi(1-e^{\pi})(b_1^1+10a_1^2)},
\end{array}
\end{equation}
 and $D=-4A_1^3-27A_0^2$. 
\begin{itemize}
\item[(i)] If $D>0$,  then the piecewise smooth vector field admits
at least one limit cycle. Moreover, if $A_1<0$ and $A_0>0$, then the
piecewise smooth vector field admits at least two limit cycles;

\item[(ii)] If $D \leq 0$ and $A_0<0,$ then the piecewise smooth vector
field admits at least one limit cycle.
\end{itemize}
Moreover, in both cases we have a limit cycle $\f(t,\e)$ of $X$ of
period $T_{\e}$ such that $T_{\e}=2\pi+\CO(\e)$,
$\f(t,\e)=(x_n^*,y_n^*,2 (x_n^*)^2-(y_n^*)^2)+\CO(\e)$, and $|(x_n^*,z_n^*)|=r_n^*$.
\end{corollary}

\begin{proof}
For $f(x,y)=2 x^2-y^2$ the bifurcation function \eqref{bifu2} becomes
\begin{equation}
\label{bifu2example}
\begin{split}
f_2(r)&=-2 b_2+\frac{a_0 \left(\left(e^{\pi } (1-\pi)+1+\pi\right)
a_1-\left(1+e^{\pi }\right) b_1\right)}{e^{\pi }-1}r\\
&+\frac{\left(-\left(e^{\pi } (56-50 \pi)+56+50 \pi \right) a_1^2+60
\left(1+e^{\pi }\right) a_1 b_1-\left(e^{\pi } (4-5 \pi)+4+5 \pi
\right) b_1^2\right)}{40 \left(e^{\pi }-1\right)}r^3.
\end{split}
\end{equation}
Dividing $f_2$ by $a_1^2+b_1^2 \neq 0$, we see that the equation $f_2(r)=0$ is equivalent to 
$\tilde{f}_2(r)\doteq A_0+A_1 r+r^3=0$, where $A_0$ and $A_1$ are provided in \eqref{A0A1}.

Notice that $\tilde{f}_2(r)$ is a polynomial function of degree $3$, so it has at least one real root and can be written as $\tilde{f}_2(r)=r^3-(r_1+r_2+r_3)
r^2+(r_1r_2+r_1r_3+r_2r_3)r-r_1r_2r_3$, where $r_i, i=1,2,3$ are the
zeros of the polynomial. Moreover, the sign of its discriminant
$D=-4A_1^3-27A_0^2$ carries information about its number of real
roots. 

If $D>0$ the polynomial $\tilde{f}_2(r)$ has three simple real roots
$r_1, r_2$ and $r_3$. Since the polynomial has no quadratic term, it
follows that $r_1+r_2+r_3=0$ and then at least one of these roots
must be positive. Moreover, if $A_1<0$ and $A_0>0$, then there are two
changes of sign between the terms of the polynomial, and then by
\textit{Descartes Sign Theorem} we get the two positive roots.

If $D \leq 0$, then there is a pair of complex roots or a double real
root. In both cases, the condition $A_0<0$ implies that at least one
root is positive.

Now,  according to Theorem \ref{maintheorem}, each positive  simple
zero of \eqref{bifu2example} provides an isolated periodic solution
of system \eqref{X1X2}. This concludes this proof.
\end{proof}

\subsection{non-smooth perturbation of a non-smooth center}

In this example, we consider a non-smooth differential system in
$\R^3$ defined in $4$ zones ($n=4$). Consider the non-smooth vector
field
\begin{equation}
\label{linearcenter} X(u,v,w)=
\begin{cases}
X_1(u,v,w)& \text{if $u>0$ and $v>0$},\\
X_2(u,v,w)& \text{if $u<0$ and $v>0$},\\
X_3(u,v,w)& \text{if $u<0$ and $v<0$},\\
X_4(u,v,w)& \text{if $u>0$ and $v<0$},
\end{cases}
\end{equation}
where
\begin{equation*}
\begin{split}
&X_1(u,v,w)=\left(-1 +  \e(a_{1} x + b_{1}),\, 1, \, -w+  \e(c_1 x + d_1)\right),\\
&X_2(u,v,w)=\left(-1 + \e(a_{2} x + b_{2}), \, -1,  \,-w+ \e(c_2 x + d_2)\right),\\
&X_3(u,v,w)=\left(1 + \e(a_{3} x + b_{3}), \,-1, \, -w+ \e (c_3 x + d_3)\right),\\
&X_4(u,v,w)=\left(1 +  \e (a_{4} x + b_{4}), \,1, \,-w+  \e (c_4 x +
d_4)\right),
\end{split}
\end{equation*}
with $a_{j}, b_{j}, c_{j}, d_{j} \in \R$ for all $j$.

Writing in cylindrical coordinates $u=r \cos \T$, $v=r \sin \T$,
$w=w,$ the discontinuity set is $\Sigma= \{\T=0\} \cup \{\T=t_1\}
\cup \{\T=t_2\} \cup \{\T=t_3\}$ with $t_0=0, t_1=\pi/2, t_2=\pi,
t_3=3\pi/2$ and $t_4=2\pi$. For each $j=1,2,3,4$ the differential
system $(\dot u,\dot v, \dot w)=X_j(u,v,w)$ in cylindrical
coordinates writes
\begin{equation*}
\label{linearcenterpolarj2}
\begin{split}
& r'(t)=g_j(\T) + \sum_{i=1}^k\e^i(a_{ij}r\cos^2\T + b_{ij}\cos \T),\\
& w'(t)=-w+\sum_{i=1}^k\e^i(c_{ij}r\cos\T + d_{ij}\cos \T),\\
& \T'(t)=\frac{1}{r}\left(\widehat{g}_j(\T) -
\sum_{i=1}^k\e^i(a_{ij}r\cos \T \sin \T + b_{ij} \sin \T)\right),
\end{split}
\end{equation*}
where
\begin{equation*}
\begin{array}{RRL}
&g_1(\T)=\sin \T - \cos \T, \quad \quad \quad &\widehat{g}_1(\T)=\sin \T + \cos \T,\\
&g_2(\T)=-(\sin \T + \cos \T), \quad \quad \quad &\widehat{g}_2(\T)=\sin \T - \cos \T,\\
&g_3(\T)=-\sin \T + \cos \T, \quad \quad \quad &\widehat{g}_3(\T)=-(\sin \T + \cos \T),\\
&g_4(\T)=\sin \T + \cos \T, \quad \quad \quad
&\widehat{g}_4(\T)=-\sin \T + \cos \T.
\end{array}
\end{equation*}
Notice that, for each $j=1,2,3,4$ and $t_{j-1}\leq\T\leq t_j$, $\dot \T(t)\neq0$ for $|\e|$ sufficiently small. Thus, in a
sufficiently small neighborhood of the origin we can take $\T$ as the
new independent time variable by doing $r'(\T)=\dot r(t)/\dot \T(t)$
and $w'(\T)=\dot w(t)/\dot \T(t)$. Taking $\T$ as the new independent
time variable, we have
\begin{equation}
\label{eqcilindrica}
\begin{split}
r'(\T)=F_{01}^j(\T,z)+\e F_{11}^j(\T,z)+\CO_1(\e^2),\\
w'(\T)=F_{02}^j(\T,z)+\e F_{12}^j(\T,z)+\CO_2(\e^2).
\end{split}
\end{equation}
Here, $z=(r,w)$ and the prime denotes the derivative with respect to
$\T$. The expressions of $F_{01}^j$ and
$F_{02}^j$ for $j=1,2,3,4$ are provided by
\begin{equation*}
\begin{split}
F_{01}^1=\dfrac{r (\sin \theta -\cos \theta )}{\sin \theta +\cos
\theta }, \,\, F_{02}^1=\dfrac{-r w}{\sin \theta +\cos \theta }, \,\,
F_{01}^2=\frac{r (\sin \T+\cos \T)}{\cos \T-\sin \T}, \,\,
F_{02}^2=\frac{r w}{\cos \T-\sin \T},\\
\vspace{0.3cm}\\
F_{01}^3=\frac{r (\sin \T-\cos \T)}{\sin \T+\cos \T}, \,\,
F_{02}^3=\frac{r w}{\sin \T+\cos \T}, \,\, F_{01}^4=\frac{r (\sin
\T+\cos \T)}{\cos \T-\sin \T}, \,\, F_{02}^4=\frac{-r w}{\cos \T-\sin
\T}.\\ \vspace{0.3cm}
\end{split}
\end{equation*}
The expressions of $F_{11}^j$
and $F_{12}^j$  for $j=1,2,3,4$ are also easily computed. Nevertheless, we shall omit these expressions due to their size.

For each $j \in \{1,2,3,4\},$ the differential system
\eqref{eqcilindrica} is $2\pi$-periodic in the variable $\T$ and 
is written in standard form with
\[
F_i^j(\T,z)=\left(F_{i1}^j(\T,z), F_{i2}^j(\T,z)\right),
\]
for $i=0,1$. Now, for each $j \in \{1,2,3,4\}$ we compute the solution
$x_j(\T,z,0)$ of the unperturbed system
\begin{equation*}
\label{us1} \dot r(\T)=F_{01}^j(\T,z), \quad \dot
w(\T)=F_{02}^j(\T,z).
\end{equation*}and this solution is
\begin{equation*}
\begin{array}{l}
x_1(\T,z,0)=\left(\frac{r}{\sin \T+\cos \T},w e^{-\frac{r \sin
\T}{\sin \T+\cos \T}} \right),\\
x_2(\T,z,0)=\left(\frac{-r}{\cos \T-\sin \T},w e^{-\frac{r \sin
\T}{\cos \T-\sin \T}-2 r} \right),\\
x_3(\T,z,0)=\left(\frac{-r}{\sin \T+\cos \T},w e^{-\frac{r \sin
\T}{\sin \T+\cos \T}-2 r} \right),\\ x_4(\T,z,0)=\left(\frac{r}{\cos
\T-\sin \T},w e^{-\frac{r \sin \T}{\cos \T-\sin \T}-4 r} \right).
\end{array}
\end{equation*}
We note that in each quadrant the denominators of these four
solutions never vanish.

Let $0<r_0< r_1$ be positive real numbers and consider the set
$\mathcal{Z} \subset \R^2$ such that $\mathcal{Z} =\{(\al,0) :
r_0<\al<r_1 \}$. The solution $x(\T,z,0)$ of the unperturbed system
$x'(\T)=F_0(\T,z)$ satisfies $x(\T,z,0)=x_j(\T,z,0)$, for $\T \in
[t_{j-1},t_j]$, and $x(2\pi,z,0)-x(0,z,0)=(0,z(1-e^{-4 r}))$. Consequently, for each $z_{\al} \in \mathcal{Z}$, the solution
$x(\T,z,0)$ is $2 \pi$-periodic and system \eqref{linearcenter}
satisfies hypothesis \text{$(H_{a}$)}. Moreover, the fundamental matrix $Y(\T,z)$
is provided by
\begin{equation*}
\label{relation2} Y(\T,z)=\begin{cases}
Y_1(\T,z) & \text{if $0=t_0\leq \T \leq \pi/2$},\\
Y_2(\T,z) & \text{if $\pi/2\leq \T \leq \pi$},\\
Y_3(\T,z) & \text{if $\pi\leq \T \leq 3\pi/2$},\\
Y_4(\T,z) & \text{if $3\pi/2\leq \T \leq 2\pi$},\\
\end{cases}
\end{equation*}
where $Y_j(t,z)$ are defined by \eqref{defined1}. Thus
\begin{equation*}
\begin{array}{RL}
Y_1(\T,z)=&\left(
\begin{array}{cc}
\frac{1}{g_4(\T)} & 0 \\
-\frac{e^{-\frac{r \sin \T}{g_4(\T)}} w \sin \T}{g_4(\T)} &
e^{-\frac{r \sin \T}{g_4(\T)}} \\
\end{array}
\right),\\
Y_4(\T,z)=&\left(
\begin{array}{cc}
\frac{1}{g_3(\T)} & 0 \\
-\frac{e^{-\frac{r \sin \T}{g_3(\T)}-4 r} w \left(\sin \T+4
g_3(\T)\right)}{g_3(\T)} & e^{-\frac{r \sin \T}{g_3(\T)}-4 r} \\
\end{array}
\right).\\
\end{array}
\end{equation*}
Hence,
\[
Y_1(0,z)^{-1}-Y_4(2 \pi ,z)^{-1}=\left(
\begin{array}{cc}
0 & 0 \\
-4 w & 1-e^{4 r} \\
\end{array}
\right),
\]
and then $\det(\Delta_{\al})=1-e^{4 r} \neq 0$ if $z_{\al}=(\al,0)
\in \mathcal{Z}$. Thus, we can compute the bifurcation functions
\eqref{bifurcationfunction} for system \eqref{linearcenter}. For
such we first obtain the functions \eqref{w11} corresponding to
this system,
\begin{equation*}
\begin{array}{cl}
g_0(\T,z)=&(0,w(1-e^{4 r})),\\
w_1^4(2\pi,z)=&\Big(\dfrac{1}{2} r (r (a_1+a_2+a_3+a_4)+2 (b_1-b_2-b_3+b_4)),\\
&\dfrac{1}{3} e^{-4 r} (-r^2 w (6 a_1+3 a_2+2 a_3)-3 r (w (4 b_1-2 b_2-b_3)\\
&+e^{2 r}(-e^{2 r} c_4+c_2+c_3)+c_1)+3 (e^r-1) (e^r (c_2+d_2)\\
&+e^{2 r} (c_3-d_3)+e^{3 r} (d_4-c_4)+c_1+d_1))\Big),
\end{array}
\end{equation*}and
\begin{equation}
\label{g1} g_1(z)=Y_4(2 \pi ,z)^{-1} w_1^4(2\pi,z).
\end{equation}
This way, the bifurcation function \eqref{bifurcationfunction} corresponding
to the function \eqref{g1} becomes
\begin{equation*}
f_1(\al)=\frac{1}{2} \al (\al (a_1+a_2+a_3+a_4)+2 (b_1-b_2-b_3+b_4)),
\end{equation*}
which has a simple zero $\al^{\ast}$. So, from Theorem
\ref{maintheorem}, we get the existence of an isolated periodic
solution of system \eqref{eqcilindrica} for $\e$ sufficiently small.

\section{Conclusion}

Establishing the existence of invariant sets, in particular equilibria and periodic solutions, is very important for  understanding the dynamics of a differential system. The detection of periodic solutions is far more complicated then equilibria and relies on solving differential equations, which cannot be done in general. For smooth perturbative differential systems, the Melnikov method and the averaging theory are two classical tools that reduce the problem of finding periodic solutions into a simpler problem of finding zeros of a related function. These tools, particularly the averaging theory, have been  recently developed for non-smooth perturbative systems.

This manuscript is devoted to develop the averaging theory for a class of non-smooth differential systems. Results relating isolated periodic solutions with zeros of a sequence of bifurcation functions are shown. In Section \ref{andronov}, we show that our results can be used in a wide family of differential systems with applied interests, namely, non-smooth perturbations of vector fields admitting Andronov-Hopf equilibria.

A natural next step for further investigations consists in considering classes of non-smooth differential systems with more general set of discontinuities. For these cases, it has been noticed in \cite{LliMerNovJDE2015} that the higher order averaged functions does not control in general the bifurcation of periodic solutions. A complete second order analysis was performed in \cite{BasBuzLliNov18}. It was shown that the second order bifurcation (Melnikov) function is provided as the second-order averaged function added to an increment, which depends on the geometry of the discontinuity set. Currently, the theory lacks a higher order analysis for these cases.

\section*{Acknowledgements}

We thank the referees for their comments and suggestions which help us
to improve the presentation of this paper.

\smallskip

The authors thank Espa\c{c}o da Escrita -- Pr\'{o}-Reitoria de Pesquisa -- UNICAMP for the language services provided.

\smallskip

JL is partially supported by the Ministerio de
Econom\'ia, Industria y Competitividad, Agencia Estatal de
Investigaci\'on grant MTM2016-77278-P (FEDER), the Ag\`encia de
Gesti\'o d'Ajuts Universitaris i de Recerca grant 2017 SGR 1617,
the European project Dynamics-H2020-MSCA-RISE-2017-777911. DDN is partially supported by FAPESP grants 2018/16430-8, 2018/ 13481-0, and 2019/10269-3, by CNPq grants 306649/2018-7 and 438975/ 2018-9. JL and DDN are also partially supported by the European Community grants  FP7-PEOPLE-2012-IRSES-316338 and FP7-PEOPLE-2012-IRSES-318999.

\end{document}